\documentclass[12pt]{amsart}
\usepackage{pdfsync}
\usepackage[latin1]{inputenc}
\usepackage{amsmath,amssymb}
\usepackage{colortbl}
\usepackage[english]{babel}
\usepackage [T1] {fontenc}
\newtheorem{thm}{Theorem}[section]
\newtheorem{cor}[thm]{Corollary}
\newtheorem{lem}[thm]{Lemma}
\newtheorem{prop}[thm]{Proposition}
\theoremstyle{definition}
\newtheorem{defn}[thm]{Definition}
\newtheorem{rem}[thm]{Remark}

 \numberwithin{equation}{section}
\newcommand\A{\mathcal{A}}

\newcommand\E{\mathcal{E}}
\newcommand\F{\mathcal{F}}

\newcommand\R{{\mathbb R}}

\newcommand\calH{\mathcal{H}}
\newcommand\K{\mathcal{K}}
\newcommand\I{\mathcal{I}}

\newcommand\B{\mathcal{B}}
\newcommand\M{\mathcal{M}}

\newcommand{\e}{\varepsilon}

\def\undemi{\frac{1}{2}}
\def\ds{\displaystyle}

\begin{document}

\title{Quantum differentials of Spectral triples, \\ Dirichlet spaces and discrete groups}
\author{Fabio E.G. Cipriani, Jean-Luc Sauvageot}
\address{(F.E.G.C.) Politecnico di Milano, Dipartimento di Matematica,
piazza Leonardo da Vinci 32, 20133 Milano, Italy.} \email{fabio.cipriani@polimi.it}
\address{(JLS) Institut de Math\'ematiques de Jussieu -- Paris Rive Gauche, CNRS -- Universit\'e Paris Cit\'e, F-75205 Paris Cedex 13, France}
\email{jean-luc.sauvageot@imj-prg.fr}
\footnote{This work has been supported by Laboratoire Ypatia des Sciences Math\'ematiques C.N.R.S. France - Laboratorio Ypatia delle Scienze Matematiche I.N.D.A.M. Italy (LYSM).}
\subjclass{58B34, 31C25, 46L57, 47A11}
\date{December 30, 2022}
\begin{abstract}
We study natural conditions on essentially discrete spectral triples $(\A,h,D)$ by which the quantum differential $da$ of $a\in\A$ belongs to the ideal generated by the unit length $ds=D^{-1}$. We also study upper and lower bounds on the singular values of the $da$'s and apply the general framework to natural spectral triples of Dirichlet spaces and, in particular, those on dual of discrete groups arising from negative definite functions.
\end{abstract}
\maketitle

\section{Introduction}

In Noncommutative Geometry \cite{Co} a spectral triple $(\A,h,D)$ is made by a $*$-algebra of bounded operators $\A\subseteq B(h)$ on a Hilbert space $h$ on which a densely defined self-adjoint (Dirac) operator $D$ also acts in such a way that the commutators $[D,a]\in B(h)$ are bounded for all $a\in\A$. The {\it unit length} or {\it line element} of $(\A, h, D)$ (\cite{Co} Chapter 7), given by the compact operator
\[
{\bf ds}=D^{-1},
\]
is meant to encode the NC metric aspects of the structure. The volume element, given by the spectral density operator ${\bf dv}=\rho(D)$ \cite{CS3}, where $\rho(x):=N_{|D|}(x)^{-1}$ is the reciprocal of the counting function of $|D|$, represents the NC measure theoretic aspects. It can be written as ${\bf dv}=\sigma ({\bf ds})$ where $\sigma(x):=N_{|D|}(x^{-1})^{-1}$ is the spectral density function of the compact operator $\bf ds$. Other fundamental infinitesimals (i.e. compact operators) are the quantum differentials $da:=i[F,a]$ of elements $a\in\A$, where $(\A,h,F)$ is a Fredholm module associated to the spectral triple. In the classical case of the Riemannian geometry of a compact smooth manifold, where $\A$ is the algebra of smooth functions,
$D$ is the Dirac operator and $F$ is the Hilbert transform, connections among the $da$'s, $\bf ds$ and their singular values, can be established thanks to the tools of the pseudo-differential calculus. Even in this commutative setting, however, a much finer analysis is needed when one only requires $\A$ to be an algebra of Lipschitz or Sobolev functions, as in the case of the one dimensional fractals and quasi-conformal manifolds (cf. works of A. Connes-D. Sullivan in \cite{Co} and M. Hilsum in \cite{Hil}). As a rule, the higher smoothness of an element $a\in\A$ has to be read by the stronger compactness of its quantum differential $da$ which, in ultimate analysis, will lie in the ideal generated by the unit length $\bf ds$.\\
In this work we estimates the singular values of the $da's$ in terms of those of $\bf ds$ without imposing any spectral growth condition on the Dirac operator $D$ and only requiring minimal and natural smoothness: $\A$ is a subalgebra either in the domain $\A_D$ of the generator of the automorphisms group $\alpha(a,t):=e^{itD}ae^{-itD}$, or it is lying in the intersection $\A_D\cap\A_{|D|}$ with the domain $\A_{|D|}$ of the generator of the automorphisms group $\beta(a,t):=e^{it|D|}ae^{-it|D|}$ (in other words, the commutators with $D$ or $|D|$ are bounded). In Quantum Mechanics $\alpha$ represents the Heisenberg time evolution of observables in a system governed by the Dirac Hamiltonian $D$ while, in Riemannian Geometry, $\beta$ represents the geodesic flow, by the Egorov Theorem.\\
In order to deal with more examples, among which the forthcoming ones, we allow the Dirac operator to have an infinite dimensional kernel, while it is assumed to have discrete spectrum away from its kernel, i.e. to have compact inverse $D^{-1}$ on the orthogonal $ker(D)^\perp$ of its kernel.

To the range of application of NCG, we add, with this work, Dirichlet spaces. These are C$^*$ or von Neumann algebras, commutative or not, with a privileged quadratic form $\E$ satisfying a characteristic contraction property and named Dirichlet form, representing an energy functional. This functional generalizes the Dirichlet integral of a Riemannian manifold and allows to develop a kernel-free NC Potential Theory. In a Dirichlet space, the Dirac operator is a $2\times 2$ matrix
\[
D=\begin{pmatrix} 0 & \partial^* \\ \partial & 0 \end{pmatrix}
\]
defined by the derivation $\partial$, which represents the {\it differential square root} of the Dirichlet energy form (cf. \cite{CS1}) in the sense that
\[
\E[a]=\|\partial a\|^2_\mathcal{H}.
\]

Natural examples come from ground state representations of Hamiltonian in Quantum Mechanics, completely positive, Markovian semigroups converging to KMS equilibria in Quantum Statistical Mechanics and generators of quantum Levy processes in Quantum Probability. In the final part of the work we study, in particular, quantum differentials of Dirichlet spaces associated to negative definite functions on countable, discrete groups. 

A somewhat detailed account of the content of the work is as follows. In Section 2, we introduce {\it essentially discrete spectral triples} $(\A,h,D)$, a slight enlargement of the classical notion which  will be needed for applications to Dirichlet spaces. In the section, we consider a canonical Fredholm operator $F_0$ and a technical variant $F$ of it, associated to any of these triples, showing that they furnish equivalent Fredholm modules on $\A$. The analysis is based on old and new representations of the quantum differentials $da$ in terms of the gradient $i[D,a]$ and the line element $\bf ds$. In particular, it is shown that the $da$'s belong to the ideal generated by $\sqrt{|{\bf ds}|}$ and that the singular values $\mu_{4k}(da)$ are controlled by the singular values $\mu_k({\bf ds})$ with ratios proportional to the Lipschitz seminorms $\|[D,a]\|$. Assuming that also the commutators $[|D|,a]$ are bounded, we prove that the $da$'s belong to the ideal generated by $\bf ds$ and that the $\mu_{2k}(da)$'s are dominated by the $\mu_k({\bf ds})$'s. Their ratios are shown to be asymptotically close to five times the sum of the seminorms $\|[D,a]\|,\|[|D|,a]\|$. In case of discrete spectral triples the estimate is improved to control $\mu_{k+d}(da)$ where $d:={\rm dim\, ker}(D)$. \\
In Section 3, we introduce the canonical, essentially discrete spectral triple and Fredholm module of a Dirichlet space $(\E,\F)$ on a C$^*$-algebra with l.s.c. faithful trace $(A,\tau)$. The Dirac operator is here the anti-diagonal matrix of the derivation $\partial$ and divergence $\partial^*$ canonically associated to $\E$, and $\A$ is the {\it Lipschitz algebra} $\A_\E$ made by elements of the Dirichlet algebra $\B:=\F\cap A$ which have bounded {\it carr\'e du champ} (alias energy density). Elements of the {\it smooth sub-algebra} $\A_\E^{2,\infty}\subseteq \A_\E$, which belong to the domain of the generator $L$ of $\E$ on the von Neumann algebra $M:=L^\infty(A,\tau)$, are shown to have bounded commutators with $|D|$. These results are consequences of the compactness on $L^2(A,\tau)$ of the commutators with the roots $(I+L)^\gamma $  when $\gamma\in (0,1/2)$ and their boundedness when $\gamma=1/2$.\\
In Section 4 we prove lower bounds on the singular values of a quantum differential $da$ when the commutator $[\sqrt{I+L},a]$ is compact. In this case, the singular values of the $da$'s are controlled {\it both above and below} by those of the line element $\bf{ds}$.
The case of Dirichlet forms whose generator are roots $L^\beta$ by $\beta\in (0,1)$ of $L$ is especially considered.\\
In the final Section 5, we analyze Dirichlet forms constructed by negative type functions on discrete groups. We introduce the class of {\it slow negative type functions} to which the entire above analysis apply, showing also that it contains the length function of free groups and any root of any proper negative type function on any discrete group.

\section{Essentially discrete spectral triples and their Fredholm modules }

To cover applications to Dirichlet spaces, we consider spectral triples of the following type:
\begin{defn}
$(\A,h,D)$ is said {\it essentially discrete} if $\sigma(D)\setminus\{0\}$ is discrete.
\end{defn}
Thus, if $0\in\sigma(D)$, this value is allowed to be an eigenvalue with infinite degeneracy i.e., denoting by $P_0$ the orthogonal projection onto ${\rm ker\,}(D)$, we are assuming that $(I-P_0)D$ has discrete spectrum as an operator on $(I-P_0)(h)$.
\vskip0.1truecm\noindent
We shall adopt the convention by which $D^{-1}$ and $|D|^{-1}$ denote the operators which identically vanish on ${\rm ker\,}(D)=P_0h$ and are the usual functional calculi on $(I-P_0)(h)$. With this convention, both these operators are compact and we have the identities
\[
DD^{-1}=D^{-1}D =|D|\,|D|^{-1}=|D|^{-1}|D|=I-P_0,\qquad D|D|^{-1}=sign(D)
\]
where the last one is the sign operator corresponding to the sign function on $\R$.
\smallskip
The nonzero part of the spectrum of $|D|$ will be enumerated as
\[
0<\lambda_1(|D|) \leq \lambda_2(|D|) \cdots \leq \lambda_n(|D|) \leq \lambda_{n+1}(|D|)\leq \cdots
\]
where each eigenvalue $\lambda_n(|D|)$ is repeated according to its multiplicity.

\subsection{Fredholm operators and quantum differentials}\label{Fred}

In the commutative situation where $D$ is the Dirac operator on the Clifford algebra of a closed, compact, Riemannian manifold $M$, the sign operator $F_{cl}:=sign(D)$ is a $0$-order $\Psi DO$, since $D$ is differential operator and $|D|$ is a $\Psi DO$ both of order $1$. By the pseudo-differential calculus, the commutators $[F_{cl},a]$ with smooth functions $a\in C^\infty(M)$ are $\Psi DO$ of order $-1$ so that $[F_{cl},a]|D|$ and $|D|[F_{cl},a]$ are bounded as $0$-order $\Psi DO$. In other words, the commutators $[F_{cl},a]$ with smooth functions belong to the symmetric principal ideal (of compact operators) generated by the line element ${\bf ds}=|D|^{-1}$. This property cannot be derived if the functions $a$ are just Lipschitz because in these cases pseudo-differential calculus does not apply. The purpose of this section is to get similar estimates in the general noncommutative geometrical workframe under rather mild assumptions on $a\in\B(h)$ of Lipschitz nature.
\vskip0.2truecm\noindent
To consider the general situation, we associate to the Dirac operator, the following self-adjoint, bounded operator
\[
F:=P_0+\frac{D}{\sqrt{1+D^2}}
\]
which is in fact a Fredholm one, as it follows from
\[
F^2-I=P_0+(I-P_0)\frac{D^2}{I+D^2}-I_h=-(I-P_0) \frac{I}{I+D^2}.
\]
An alternative self-adjoint, Fredholm operator will be also considered:
\[
F_0:=P_0+D\,|D|^{-1}.
\]
It is an orthogonal symmetry which is equal to $F$ up to a compact operator. More precisely
\begin{equation}\label{F-F0}
F^2=I,\qquad F-F_0=(I-P_0)T\,|D|^{-2}
\end{equation}
where $\displaystyle T:=-\frac{D}{|D|}\frac{D^2}{\sqrt{I+D^2}(|D|+\sqrt{I+D^2})}$ is a contraction.
\vskip0.2truecm\noindent
The operators $i[F_0,a]$, $i[F,a]$, both denoted by $da$, are the {\it quantum differentials} of $a\in\A$ (\cite{Co} Chapter 4).
Let us start with an observation:

\begin{lem}\label{P0}
The following identity holds true
\[
[P_0,a]=P_0\, \alpha_0\,|D|^{-1}+|D|^{-1}\beta_0\,P_0\qquad a\in\A
\]
for suitable operators $\alpha_0,\beta_0\in B(h)$ with $\|\alpha_0\|= \|\beta_0\|=\|[D,a]\|$.
\end{lem}
\begin{proof}
The stated representation follows from the identities
\[
\begin{split}
P_0a-P_0aP_0&=P_0a(I-P_0)=P_0aDD^{-1}=P_0[a,D]D^{-1}=P_0[a,D]|D|D^{-1}|D|^{-1}\\
P_0aP_0-aP_0&=(I-P_0)aP_0=D^{-1}DaP_0=D^{-1}[D,a]P_0=|D|^{-1}|D|D^{-1}[D,a]P_0.
\end{split}
\]
\end{proof}

\begin{prop}\label{cruc1} For $a=a^*\in\A$ we have the bounds
\begin{equation}\label{cruc22}
-||\,[D,a]\,||\, (I+D^2)^{-1/2} \leq i\,\big[ \frac{D}{\sqrt{1+D^2}},\,a\big] \leq ||\,[D,a]\,||\, (I+D^2)^{-1/2}
\end{equation}
\begin{equation}\label{cruc23}
-||\,[D,a]\,||\, |D|^{-1} \leq i\,(I-P_0)\,\big[ \frac{D}{\sqrt{1+D^2}},\,a\big]\,(I-P_0) \leq ||\,[D,a]\,||\, |D|^{-1}\\
\end{equation}
and
\begin{equation}\label{cruc24}
-\|[D,a]\|\cdot |D|^{-1}\le i\,(I-P_0)\,\big[\,\frac{D}{|D|},\,a\,\big]\,(I-P_0) \le\|[D,a]\|\cdot |D|^{-1}.
\end{equation}
\end{prop}

\begin{proof}
Notice that the double inequality (\ref{cruc23}) is a straightforward consequence of (\ref{cruc22}), since $\displaystyle (I-P_0)\frac{1}{\sqrt{1+D^2}}(I-P_0)\leq |D|^{-1}$.
In order to prove the double inequality in (\ref{cruc22}), we use, as usual (see \cite{SWW} Proposition 1), the identity (in the strongly convergent sense)
$$\frac{D}{\sqrt{I+D^2}}=\frac{1}{\pi}\int_0^{+\infty}t^{-1/2} \frac{D}{tI+I+D^2}dt$$
from which we deduce
\begin{equation*}\begin{split}
 i\big[ \frac{D}{\sqrt{1+D^2}},\,a\big]&=\frac{1}{\pi}\int_0^{+\infty}t^{-1/2}\,i\,\Big[\frac{D}{tI+I+D^2},\,a\,\Big]dt\\
 &=\frac{1}{\pi}\int_0^{+\infty}t^{-1/2}\left(\frac{I}{tI+I+D^2}i[D,a]+I\big[\frac{I}{tI+I+D^2},\,a\,\big]\,D\,\right)dt\\
 &=\frac{1}{\pi}\int_0^{+\infty}t^{-1/2}\left(\frac{I}{tI+I+D^2}i[D,a]-\frac{I}{tI+I+D^2}i[D^2,a]\frac{D}{tI+I+D^2}\right)dt \\
 &=\frac{1}{\pi}\int_0^{+\infty}t^{-1/2}\Big(\frac{I}{tI+I+D^2}i[D,a]-\frac{D}{tI+I+D^2}i[D,a]\frac{D}{tI+I+D^2} \\
&\hskip6cm -\frac{I}{tI+I+D^2}i[D,a]\frac{D^2}{tI+I+D^2}\Big)dt \\
&=\frac{1}{\pi}\int_0^{+\infty}t^{-1/2}\Big(\frac{I}{tI+I+D^2}i[D,a]\big(I-\frac{D^2}{tI+I+D^2}\big)\\
&\hskip6cm-\frac{D}{tI+I+D^2}i[D,a]\frac{D}{tI+I+D^2}\Big)dt\\
&=\frac{1}{\pi}\int_0^{+\infty}t^{-1/2}\Big(\frac{tI+I}{tI+I+D^2}i[D,a]\frac{I}{tI+I+D^2}\\
&\hskip6cm-\frac{D}{tI+I+D^2}i[D,a]\frac{D}{tI+I+D^2}\Big)dt\,.
\end{split}\end{equation*}
Noticing now that $i[D,a]$ is self-adjoint, so that $-||\,[D,a]\,||\,I \leq i[D,a]\leq ||\,[D,a]\,||\,I$, the above identity thus provides
\begin{equation*}\begin{split}
 i\big[ \frac{D}{\sqrt{1+D^2}},\,a\big]&\leq ||\,[D,a]\,||\,\frac{1}{\pi}\int_0^{+\infty}t^{-1/2}\Big(\frac{tI+I}{(tI+I+D^2)^2} +\frac{D^2}{(tI+I+D^2)^2}\Big)dt\\
 &= ||\,[D,a]\,||\,\frac{1}{\pi}\int_0^{+\infty}t^{-1/2}\frac{I}{tI+I+D^2}\,dt\\
 &=||\,[D,a]\,||\,(I+D^2)^{-1/2}
 \end{split}\end{equation*}
and, similarly,
\begin{equation*}\begin{split}
 i\big[ \frac{D}{\sqrt{1+D^2}},\,a\big]&\geq - ||\,[D,a]\,||\,\frac{1}{\pi}\int_0^{+\infty}t^{-1/2}\Big(\frac{tI+I}{(tI+I+D^2)^2} +\frac{D^2}{(tI+I+D^2)^2}\Big)dt\\
 &= -||\,[D,a]\,||\,\frac{1}{\pi}\int_0^{+\infty}t^{-1/2}\frac{I}{tI+I+D^2}\,dt\\
 &=-||\,[D,a]\,||\,(I+D^2)^{-1/2}\,.
\end{split}
\end{equation*}
The double inequality (\ref{cruc24}) can be proved the same way, with obvious adaptations, or derived from Proposition 1 in \cite{SWW} where the same inequality is stated in case $D$ is invertible. In fact, on one hand
\[
 i\,(I-P_0)\,\big[\,\frac{D}{|D|},\,a\,\big]\,(I-P_0) = i\big[\,(I-P_0)\,\frac{D}{|D|}\,(I-P_0)\,,\,(I-P_0) a(I-P_0) \,\big]= i\big[\,\frac{D}{|D|},\,(I-P_0) a(I-P_0) \,\big]
 \]
and, since  $\|\big[\,D,\,(I-P_0) a(I-P_0) \,\big]\|=\|[D,a]\|$, $(D,(I-P_0)h,(I-P_0)\A(I-P_0))$ is a spectral triple where $D$ is invertible on $(I-P_0)h$.
\end{proof}
Using the bounds above, we show that  the quantum differentials of elements of $\A$ belong to the symmetric principal ideal in $B(h)$ generated by $\sqrt{|{\bf ds}|}=|D|^{-1/2}$.

\begin{prop}\label{Fa1}
For any fixed $a=a^*\in \A$
\vskip0.1truecm\noindent
i) there exist bounded operators $\alpha, \beta, \gamma\in B(h)$ such that
\[
i[F,a]=\alpha |D|^{-1}+|D|^{-1}\beta + |D|^{-1/2} \gamma\, |D|^{-1/2},
\]
with $||\alpha||\leq 2\,||\,[D,a]\,||$, $||\beta||\leq 2\,||\,[D,a]\,||$, $||\gamma||\leq ||\,[D,a]\,||$;
\vskip0.1truecm\noindent
ii) a similar representation holds true with $F_0$ in place of $F$.
\end{prop}
\begin{proof}
i) For $[P_0,a]$, see Lemma \ref{P0} above. Since $DP_0=0$ we have
\[
\begin{split}
P_0[(1+D^2)^{-1/2}D,a]&=-P_0aD(1+D^2)^{-1/2}(I-P_0)=P_0[D,a](1+D^2)^{-1/2}(I-P_0) \\
&=P_0[D,a]\frac{|D|}{\sqrt{1+D^2}}|D|^{-1}=P_0[D,a]\frac{|D|}{\sqrt{1+D^2}}|D|^{-1}(1-P_0).
\end{split}
\]
Similarly,
\[
[(1+D^2)^{-1/2}D,a]P_0=(1-P_0)|D|^{-1}\frac{|D|}{\sqrt{1+D^2}}\,[D,a]P_0.
\]
As for $(I-P_0)[(1+D^2)^{-1/2}D,a](I-P_0)=(I-P_0)[(1+D^2)^{-1/2}D,(I-P_0)a(I-P_0)](I-P_0)$, one invokes  (\ref{cruc23}) of Proposition \ref{cruc1}, which is equivalent to the assertion
\[
-||\,[D,a]\,||\,|D|^{-1} \leq i\big[\frac{D}{\sqrt{I+D^2}},(I-P_0)a(I-P_0)\big]\leq -||\,[D,a]\,||\,|D|^{-1}\,.
\]
ii) follows from (\ref{cruc24}) in Proposition \ref{cruc1} and the same proof as above.
\end{proof}
As first consequence of the above representation, we have
\begin{cor}
If $(\A,h,D)$ is an essentially discrete spectral triple, then $(\A,h,F_0)$ and $(\A,h,F)$ are, essentially unitary equivalent, Fredholm modules.
\end{cor}
\begin{proof} Follows from the identity (\ref{F-F0}), Lemma \ref{P0} and Proposition \ref{cruc1}.
\end{proof}
A second consequence concerns a bound on the singular values of the quantum differentials.
\begin{cor}\label{muk}
If $(\A,h,D)$ is an essentially discrete spectral triple, the singular values of quantum differentials are controlled by the Lipschitz seminorm and the singular values of $D^{-1}$
\[
\mu_{4k}(i[F_0,a]) \leq 5 \,\|\,[D,a]\,\|\,\mu_k(|D|^{-1})=5 \,\|\,[D,a]\,\|\,\lambda_{k+1}(|D|)^{-1}\qquad a=a^*\in\A,\quad k\ge 0,
\]
\[
\mu_{4k}(i[F,a]) \leq 5 \,\|\,[D,a]\,\|\,\mu_k(|D|^{-1})=5 \,\|\,[D,a]\,\|\,\lambda_{k+1}(|D|)^{-1}\qquad a=a^*\in\A,\quad k\ge 0.
\]
In terms of line element and quantum differentials, we proved the bounds
\[
\mu_{4k}(da) \leq 5 \,\|\,[D,a]\,\|\,\mu_k({\bf ds})\qquad a=a^*\in\A,\quad k\ge 0.
\]
\end{cor}
\begin{proof} Applying Proposition \ref{Fa1} and the rules of singular values \cite{Co} Chapter 4 Appendix C, we have
\begin{equation*} \begin{split}
\mu_{4k}(i[F,a])&\leq \mu_k(\alpha |D|^{-1})+\mu_k(|D|^{-1}\beta)+\mu_{2k}( |D|^{-1/2} \gamma\, |D|^{-1/2})\\
&\leq \|\alpha\|\mu_k(|D|^{-1})+\|\beta\| \mu_k(|D|^{-1})+\|\gamma\|\mu_k(|D|^{-1/2})^2 \\
&\leq 5\,\|\,[D,a]\,,\mu_k(|D|^{-1})=5 \,\|\,[D,a]\,\|\,\lambda_{k+1}(|D|)^{-1}\,.
\end{split}\end{equation*}
and the same for $\mu_{4k}(i[F_0,a])$.
\end{proof}

\subsection{ Improving the estimates for the singular values of quantum differentials}

Here we consider a natural condition by which the quantum differentials $da$ of elements $a\in\A$ belong to the principal ideal in $\K(h)$ generated by the line element ${\bf ds}$. The quantum differentiation operator $d$ can then be considered as a derivation from $\A$ to the ideal $\I_D\subseteq\K(h)$ generated by ${\bf ds}$, seen as a $A$-bimodule over the norm closure $A:=\overline\A\subseteq\B(h)$.

\begin{lem}\label{C00}
Let $a\in \A$ be such that  the commutator $[\,|D|,a\,]$ is bounded. Then one has
\begin{equation} \begin{split}
[P_0,a](I-P_0)&=P_0[a,D]F_0\,|D|^{-1} \\
\Big[\,\frac{D}{|D|},a\Big](I-P_0) &=-D|D|^{-1}[\,|D|,a]\,|D|^{-1}+[D,a]\,|D|^{-1} \\
\Big[\,\frac{D}{\sqrt{I+D^2}},a\Big](I-P_0) &=-\frac{D}{\sqrt{I+D^2}}\,\big[\sqrt{I+D^2},a\big]\,\frac{|D|}{\sqrt{I+D^2}}|D|^{-1}
+[D,a]\,\frac{|D|}{\sqrt{I+D^2}}|D|^{-1}
\end{split}\end{equation}
and
\begin{equation} \begin{split} [P_0,a]P_0&=-|D|^{-1}F_0[D,a]P_0 \\
\Big[\,\frac{D}{|D|},a\Big]\,P_0&=|D|^{-1}[D,a]P_0 \\
\Big[\,\frac{D}{\sqrt{I+D^2}},a\Big]\,P_0 &=|D|^{-1}\frac{D}{\sqrt{I+D^2}}[D,a]P_0\,.
\end{split}\end{equation}
\end{lem}
\begin{proof} Let us compute successively\,:
$$[P_0,a](I-P_0)=P_0a(I-P_0)=P_0aDD^{-1}=P_0[a,D]D^{-1}=P_0[a,D]F_0|D|^{-1}$$
\begin{equation*}\begin{split}
\Big[\frac{D}{|D|},a\Big](I-P_0)&=D[\,|D|^{-1},a](I-P_0)+[D,a]\,|D|^{-1}\\
&=D(I-P_0)[\,|D|^{-1},a](I-P_0)+[D,a]\,|D|^{-1}\\
&=-D|D|^{-1}[\,|D|,a]\,|D|^{-1}+[D,a]\,|D|^{-1}
\end{split}\end{equation*}
\begin{equation*}\begin{split}
\Big[\frac{D}{\sqrt{I+D^2}},a\Big](I-P_0)&=D\Big[\,\frac{I}{\sqrt{I+D^2}},a\Big](I-P_0)+[D,a]\,\frac{I-P_0}{\sqrt{I+D^2}}\\
&=D(I-P_0)\Big[\,|D|^{-1},a\Big](I-P_0)+[D,a]\,|D|^{-1}\\
&=-D\frac{I}{\sqrt{I+D^2}}\big[\,{\sqrt{I+D^2},a\big]\frac{I-P_0}{\sqrt{I+D^2}}}+[D,a]\,\frac{I-P_0}{\sqrt{I+D^2}} \\
&=-D\frac{I}{\sqrt{I+D^2}}\big[\,{\sqrt{I+D^2},a\big]\frac{ID|}{\sqrt{I+D^2}}}\,|D|^{-1}+[D,a]\,\frac{|D|}{\sqrt{I+D^2}} \,|D|^{-1}
\end{split}\end{equation*}
$$[P_0,a]P_0=-(I-P_0)aP_0=-D^{-1}[D,a]P_0=-|D|^{-1}F_0[D,a]P_0$$
$$\Big[\frac{D}{|D|},a\Big]P_0=\frac{D}{|D|}aP_0=|D|^{-1}[D,a]P_0$$
\begin{equation*}\begin{split}
\Big[\frac{D}{\sqrt{I+D^2}},a\Big]P_0&=\frac{D}{\sqrt{I+D^2}}aP_0\\
&=\frac{I-P_0}{\sqrt{I+D^2}}\,\big[D,a\big]\,P_0 \\
&=|D|^{-1}\frac{|D|}{\sqrt{I+D^2}}\,[D,a]\,P_0\,.
\end{split}\end{equation*}
\end{proof}
\begin{prop}\label{C0} Let $a\in \A$ be such that  the commutator $[\,|D|,a\,]$ is bounded. Then there exist bounded operators $\alpha_0$ and $\beta_0$ such that
\[
i[F_0,a]=\alpha_0\,|D|^{-1}+|D|^{-1}\,\beta_0
\]
with $\|\alpha_0\| \leq \|\,[D,a]\,\|$, $\|\beta_0\| \leq \|\,[D,a]\,\|+\|\,[\,|D|,a\,]\,\|$ and $\beta_0P_0=0$. The quantum differential $da=i[F_0,a]$ thus belongs to the symmetric ideal $\I_D\subseteq\B(h)$ generated by the line element $\bf ds$.
\end{prop}
\begin{proof} This is a straightforward consequence of the previous Lemma \ref{C00}.
\end{proof}

\medskip The similar result for the commutator $[F,a]$ needs a preliminary lemma:

\begin{lem}\label{C1}
Let $a\in \A$ be such that  the commutator $[|D|,a]$ is bounded. Then
\vskip0.1truecm\noindent
i) the following bound holds true $\|[\sqrt{1+D^2},a]\,\|\leq \|[|D|,\,a]\,\|+2\|a\|$.
\vskip0.1truecm\noindent
ii) one has
\[
\|(I-P_0)\big[\sqrt{1+D^2},a\big]\,(I-P_0)\|\leq C_1(\lambda_1) \|\big[|D|,a\big]\|.
\]
where $C_1(\lambda_1)=\sqrt{1+\lambda_1^{-2}}$ and $\lambda_1:=\lambda_1(|D|)>0$ is the first nonzero eigenvalue of $|D|$.
\end{lem}
\begin{proof}
The estimate in i) follows from the identity $\ds \sqrt{1+D^2}=|D|+\frac{I}{\sqrt{1+D^2}+|D|}$ and the bound $\| (\sqrt{1+D^2}+|D|)^{-1}\|\leq 1$. The same identity also implies
 \begin{equation*}
 \begin{split}
 (I-P_0\big)[\sqrt{1+D^2},a\big](I-P_0)=&(I-P_0)\big[\,|D|,a\,\big](I-P_0) \\
 &-\frac{I-P_0}{\sqrt{1+D^2}+|D|}\,\big[\,\sqrt{1+D^2},\,a\,\big]\,\frac{I-P_0}{\sqrt{1+D^2}+|D|}\\
 &-\frac{I-P_0}{\sqrt{1+D^2}+|D|}\,\big[\,|D|,\,a\,\big]\,\frac{I-P_0}{\sqrt{1+D^2}+|D|}
 \end{split}
 \end{equation*}
 and the inequality
 \begin{equation*}\begin{split}
 \big\|\,(I-P_0)[\sqrt{I+D^2},a\,\big](I-P_0)\,\big\|\leq \big\|\,&[\,|D|,a\,\big]\big\|+ (\sqrt{1+\lambda_1^2}+\lambda_1)^{-2}\big\|\,(I-P_0)[\,|D|,a\,\big](I-P_0)\,\big\| \\
 +& (\sqrt{1+\lambda_1^2}+\lambda_1)^{-2}\big\|\,(I-P_0)[\sqrt{I+D^2},a\,\big](I-P_0)\,\big\|
 \end{split}\end{equation*}
 which in turn provides
 $$\big(1- (\sqrt{1+\lambda_1^2}+\lambda_1)^{-2}\big)\big\|\,(I-P_0)[\sqrt{I+D^2},a\,\big](I-P_0)\,\big\|\leq
 \big(1+ (\sqrt{1+\lambda_1^2}+\lambda_1)^{-2}\big)\,\|[\,|D|,a\,\big]\|
 $$
 and the result with
 \begin{equation*}\begin{split}
 C_1(\lambda_1)&=\frac{1+ (\sqrt{1+\lambda_1^2}+\lambda_1)^{-2}}{1- (\sqrt{1+\lambda_1^2}+\lambda_1)^{-2}}=\frac{1+ (\sqrt{1+\lambda_1^2}-\lambda_1)^{2}}{1- (\sqrt{1+\lambda_1^2}-\lambda_1)^{2}}=\frac{1+\lambda_1^2-\lambda_1\sqrt{1+\lambda_1^2}}{\lambda_1\sqrt{1+\lambda_1^2}-\lambda_1^2}\\
 &=\frac{1}{\lambda_1\sqrt{1+\lambda_1^2}-\lambda_1^2}-1=\frac{\sqrt{1+\lambda_1^2}+\lambda_1}{\lambda_1}-1=\sqrt{1+\lambda_1^{-2}}\,.
  \end{split}\end{equation*}
 \end{proof}
 \begin{rem}One can slightly improve estimate i) above by replacing $||a||$ by the norm of $a$ in the quotient space $A/A\cap \{D\}'$.
 \end{rem}

\begin{prop}\label{commcomm}  Let $a\in\A$ be such that  the commutator $[\,|D|,a\,]$ is bounded. Then there exist bounded operators  $\alpha_1, \beta_1\in \B(h)$ such that
\[ i[F,a]=\alpha_1|D|^{-1}+|D|^{-1}\beta_1
\]
with $\|\alpha_1\| \leq 2\|\,[D,a]\,\|$, $\|\beta_1\|\leq 3\|\,[D,a]\,\|+C_1(\lambda_1)||\,[\,|D|,a]\,\|$ (or, at choice, $\|\beta_1\|\leq 3\|\,[D,a\,]\|+\|\,[\,|D|,a]\,\|+2||a||)$
and $\beta_1=\beta_1\,P_0$ and the quantum differential $da=i[F,a]$ belongs to the symmetric ideal $\I_D\subseteq\B(h)$ generated by the line element $\bf ds$.
\end{prop}
\begin{proof} A straightforward consequence of Lemmas \ref{C00} and \ref{C1}.
\end{proof}

\begin{cor}\label{muk2}
Let $a\in\A$ be such that  the commutator $[\,|D|,a\,]$ is bounded, then singular values of the quantum differentials are controlled by
  \[
\mu_{2k}(i[F_0,a]) \leq (||\alpha_0||+||\beta_0||)\,\mu_k(|D|^{-1})=(||\alpha_0||+||\beta_0||)\,\|\,\lambda_{k+1}(|D|)^{-1}
\]
 \[
\mu_{2k}(i[F,a]) \leq (||\alpha_1||+||\beta_1||)\,\mu_k(|D|^{-1})=(||\alpha_1||+||\beta_1||)\,\|\,\lambda_{k+1}(|D|)^{-1}
\]
with $\alpha_0,\beta_0,\alpha_2$ and $\beta_1$ provided by Propositions \ref{C0} and \ref{commcomm}. In terms of line element and quantum differentials, we have
\[
\mu_{2k}(da) \leq const.\mu_k({\bf ds})\qquad a=a^*\in\A,\quad k\ge 0.
\]
\end{cor}
\begin{proof} Apply Propositions \ref{C0} and \ref{commcomm} along the lines of the proof of Corollary \ref{muk}.
\end{proof}

From this proposition, in case of discrete spectrum, we deduce the following :
\begin{cor} Let $a\in \A$, $a=a^*$ such that  commutator $[\,|D|,a]$ is bounded and suppose that the kernel of $D$ is finite dimensional. Then one has the estimates
$$\mu_{k+d}(i[F_0,a])\leq ||\alpha_0||\,\mu_k(|D|^{-1})=||\alpha_0||\,\lambda_{k+1}(|D|)^{-1}$$
$$\mu_{k+d}(i[F,a])\leq ||\alpha_1||\,\mu_k(|D|^{-1})=||\alpha_1||\,\lambda_{k+1}(|D|)^{-1}\,.$$
with $\alpha_0$ and $\alpha_1$ are provided by Propositions \ref{C0} and \ref{commcomm} respectively and $d=dim(ker(D))$.\\
In particular, $||\alpha_0||\leq ||\,[D,a]\,||$ and $||\alpha_1||\leq 2||\,[D,a]\,||$. In terms of line element and quantum differentials, we have
\[
\mu_{k+d}(da) \leq const.\mu_k({\bf ds})\qquad a=a^*\in\A,\quad k\ge 0.
\]
\end{cor}

\medskip
\subsection{The same estimates from an asymptotic point of view.}

Here we prove under the double Lipschitz assumptions above, estimates which are asymptotically a bit more precise.

\begin{prop}\label{commcomm2}  Let $a\in\A$ be such that  the commutator $[\,|D|,a\,]$ is bounded. Then there exist bounded operators  $\alpha_3, \beta_3,\gamma_3\in B(h)$ such that
\[ i[F,a]=\alpha_3|D|^{-1}+|D|^{-1}\beta_3+|D|^{-1}\gamma_3|D|^{-1}
\]
with $\|\alpha_3\| \leq 2\|\,[D,a]\,\|$, $\|\beta_3\|\leq 3\|\,[D,a]\,\|+\,||\,[\,|D|,a]\,\|$ and $\gamma_3$ bounded.
\end{prop}
\begin{proof} According to Lemma \ref{C00}, we have
$$[P_0,a]=\sigma_0|D|^{-1}+|D|^{-1}\tau_0 \text{ with } ||\sigma_0||\leq ||\,[D,a]\,||\; ||\tau_0||\leq ||\,[D,a]\,||\; $$
$$\Big[\frac{D}{\sqrt{I+D^2}},a\,\Big] P_0=|D|^{-1}\tau_1 \text{ with }||\tau_1||\leq ||\,[D,a]\,||\; $$
\begin{equation*}
\Big[\frac{D}{\sqrt{I+D^2}},a\,\Big] \big(I-P_0\big)=\sigma_1|D|^{-1}-\frac{D}{\sqrt{I+D^2}}[\sqrt{I+D^2},a]\frac{D}{\sqrt{I+D^2}}\text{ with }||\sigma_1||\leq ||\,[\,|D|,a]\,||.
\end{equation*}
We compute
\begin{equation*}\begin{split}
[\sqrt{I+D^2},a]&=[\,|D|,a]+\Big[\frac{I}{\sqrt{I+D^2}+|D|},a\Big]\\
&=[\,|D|,a]-\frac{I}{\sqrt{I+D^2}+|D|}\,\big[\,\sqrt{I+D^2}+|D|,a  \big]\frac{I}{\sqrt{I+D^2}+|D|}
\end{split}\end{equation*}
and notice that $\big[\,\sqrt{I+D^2}+|D|,a  \big]$ is a bounded operator. Summing up, we get the result.
\end{proof}

We need a Lemma which slightly ameliorates a result due to K. Fan (\cite{GK} Ch. II par. 5 Theorem 2.3). For sake of completeness we provide a detailed proof.
\begin{lem}\label{Tsigma} Let $T$ and $\sigma$ be two compact operators. Then there exist an integer $d_1$ and two sequences $(\e_k)_{k\geq 0}$ and $(\e'_k)_{k\geq 0}$ such that $\lim_{k\to \infty}\e_k=\lim_{k\to \infty} \e'_k=0$ and
$$(1-\e'_k)\mu_{k+d_1}(T) \leq \mu_k\big(T(I+\sigma)\big)\leq (1+\e_k)\mu_k(T)\,.$$
\end{lem}
\begin{proof} Let us recall ([Connes chapter 4 section 2]) that
$$\mu_k(T)=\inf ||P^\perp T||=\inf ||T\,Q^\perp||$$
where $P$ or $Q$ runs in the set of orthogonal projections with rank less than $k$, and that the infimum is indeed a minimum, reached when $P$ (resp. $Q$) is the orthogonal projection corresponding to the $k$ first larger eigenvalues of $|T^*|$ (resp. $|T|$).

\smallskip
Let $P_k$ (resp. $Q_k$) be the orthogonal projection corresponding the $k$ first eigenvalues of $|T^*|$ (resp. $|T|$)\,: we have $\mu_k(T)=||P_k^\perp T||$ and $P_kT=TQ_k$, hence $P_k^\perp T = T Q_k^\perp=P_k^\perp T Q_k^\perp$. Notice that, as $k\to \infty$, the $Q_k$ tend increasingly toward $I-q_0$, so that the $Q_k^\perp$ tend to $q_0$, where $q_0$ is the orthogonal projection on the kernel of $T$. Notice that, as $\sigma$ is compact, $\lim_{k\to \infty}||Q_k^\perp (I-q_0)\sigma||=0$. Compute now
\begin{equation*} \begin{split}
\mu_k\big(T(I+\sigma)\big)&\leq ||P_k^\perp T(I+\sigma)|| \\
&=||P_k^\perp TQ_k^\perp (I-q_0)(I+\sigma)|| \\
&\leq ||P_k^\perp T||\,\big(1+||Q_k^\perp (I-q_0)\sigma||)
\end{split}\end{equation*}
which provides the right inequality, with $\e_k=||Q_k^\perp (I-q_0)\sigma||$.

\smallskip Notice now that there exists a compact operator $\tau$ such that $(I+\sigma)(I+\tau)=I-p_1$, where $p_1$ is the orthogonal projection on $ker(I+\sigma^*)=Im(I+\sigma)^\perp$. This is a finite rank projection, with rank $d_1$. Applying the inequality just proved above, we can write
\begin{equation*}\begin{split}
\mu_{k+d_1}(T)&=\mu_{k+d_1}\big(T(I-p_1)+Tp_1\big) \\
&\leq \mu_k\big(T(I-p_1)\big)+\mu_{d_1}(Tp_1)=\mu_k\big(T(I-p_1)\big) \\
&=\mu_k\big(T(I+\sigma)(I+\tau)\big)\\
&\leq\mu_k\big(T(I+\sigma)\big)\,(1+\widetilde \e_k)
\end{split}\end{equation*}
with $\lim_{k\to \infty} \widetilde \e_k=0$. This ends the proof.
\end{proof}

\medskip
\begin{cor}  Let $a\in\A$ be such that  the commutator $[\,|D|,a\,]$ is bounded. Then one has
$$\mu_{2k}\big(i[F,a]\big)\leq \big(5||\,[D,a]\,||+||\,|D|,a]\,||\big)\,(1+\e_k)\,\lambda_{k+1}(|D|)^{-1} \text{ with } \lim_{k\to \infty} \e_k=0\,.$$
\end{cor}

\section{Spectral triples of Dirichlet spaces}
In this section we construct the spectral triple of a Dirichlet space on a C$^*$-algebra with trace and we apply the previous results to study the associated Fredholm module.
\subsection{Dirichlet forms and their tangent bimodules}
For the definition and properties of Dirichlet forms on trace C$^*$-algebras we refer to \cite{AHK}, \cite{C1}, \cite{C2}, \cite{CS1}, \cite{DL}, \cite{S}. We list below their main properties we need and fix notations for the rest of the paper.
\vskip0.2truecm
In the following, $(A,\tau)$ is a C$^*$-algebra equipped with a faithful, densely defined, lower semi-continuous trace and $M:=\pi_\tau(A)''\subseteq B(L^2(A,\tau))$ is the corresponding von Neumann algebra acting on the Hilbert space of the GNS representation $\pi_\tau:A\to B(L^2(A,\tau))$.\\
We consider a completely Dirichlet form $(\E,\F)$ on $L^2(A,\tau)$ and its densely defined, positive, self-adjoint generator $(L,D(L))$ in such a way that $\E$ is the closure of the quadratic form $D(L)\ni \xi\rightarrow (\xi|L\xi)$ and one has $\F=D(L^{1/2})$ and $\E[\xi]=\|L^{1/2}\xi\|^2$ for $\xi\in \F$.
\vskip0.1truecm
Since now on, we shall assume that $(L,D(L))$ {\it has discrete spectrum away from zero}.
\vskip0.1truecm
Among the characteristic properties of a Dirichlet form, we recall that\\
i) the semigroup $\{e^{-tL}:t>0\}$ maps $ L^2(A,\tau)\cap M$ into itself and extends to a $\sigma$-weakly continuous, completely positive contraction semigroup of $M$, still denoted by the same symbol;\\
ii) the resolvent $\{(I+tL)^{-1}:t>0\}$ maps $ L^2(A,\tau)\cap M$ into itself and extends to a $\sigma$-weakly continuous, completely positive contraction resolvent of $M$, still denoted by the same symbol.
\vskip0.1truecm
\label{DomM} The generator of the semigroup on $M$, denoted by $(L,D_M(L))$, has a domain
\[
D_M(L):=\{x\in M:\, L(x):=\lim_{t\downarrow 0}(x-e^{-tL}x)/t\quad\text{exists}\,\sigma-\text{weakly in}\, M\}
\]
which, for any $t>0$, coincides with $(I+tL)^{-1}(M)$.
\vskip0.1truecm
The Dirichlet algebra $\B:=\F\cap A$ is an involutive subalgebra of $A$. We assume that $(\E,\F)$ is {\it regular} in the sense that $\B$ is dense both in $A$ and $L^2(A,\tau)$, in their respective topologies, and that it is a form core.

\subsection{Tangent bimodule of a Dirichlet space and carr\'e du champ}
Let us recall the main results of [CS1]: to any regular Dirichlet form $(\E,\F)$ is associated a symmetric $A$-bimodule $(\calH,\mathcal{J})$ together with a symmetric derivation $\partial:\,\B\to \calH$, i.e. a linear map satisfying
\[
\partial(a^*)=\mathcal{J}(\partial a)\qquad a,b\in \B,
\]
and the Leibnitz rule
\begin{equation}\label{Leibnitz}
\partial(ab)=a\partial b+(\partial a)b\qquad a,b\in \B,
\end{equation}
which is closable as a densely defined operator from $L^2(A,\tau)$ into $\calH$, and such that
\begin{equation*}
\E[a]=\|\partial a\|^2_\calH \qquad a\in \B\,.
\end{equation*}
The {\it carr\'e du champ} or {\it energy density} of $a\in\B$ is the following positive linear form $\Gamma[a]\in A^*_+$
\[
\langle\Gamma[a],b\rangle=(\partial a| (\partial a)b)_\calH\qquad b\in A\,.
\]
A useful approximation of the carr\'e du champ (cf. \cite{CS1}) is the following one:
\begin{lem}
i) For any $\e>0$, $\ds L_\e:= \frac{L}{1+\e L}$ generates a bounded Dirichlet form on $L^2(A,\tau)$.\\
ii) Setting $\ds\Gamma_\e[a]:=\frac{1}{2}\big(a^*L_\e(a)+L_e(a)^*a-L_\e(a^*a)\big)\in M\cap L^1(A,\tau)$ for $a\in \B$, one has
\begin{equation}\label{Gamma}
\langle\Gamma[a],b\rangle =\lim_{\e \downarrow 0}\tau\big(\,\Gamma_\e[a]\,b\,\big)\qquad b\in A.
\end{equation}
\end{lem}
\begin{rem} i) A regular Dirichlet form provides the C$^*$-algebra of a potential theoretic structure which generalizes the classical Dirichlet integral on a Riemannian manifold;
\vskip0.1truecm\noindent
ii) the derivation $(\partial,\mathcal{B})$ is closable and the domain of its closure coincides with the form domain $\F$. When no confusion can arise, we shall use the same notation $\partial$ for the closure;
\vskip0.1truecm\noindent
iii) there is not, in general, a formula for the adjoint {\it divergence operator} $\partial^*$ from $\calH$ to $L^2(A,\tau)$, except in case where there exists a subalgebra $\B_0$ of $A$ contained in the domain of $L$, for which
\[
\partial^*(\partial(a)b)=\frac{1}{2}\big(L(ab)+L(a)b-aL(b)\big)\,,\;a,b\in \B_0.
\]
In the classical case of the Dirichlet integral, the derivation coincides with the gradient operator and its adjoint with the divergence operator. An example of computation of derivation and divergence when any such subalgebras $\B_0$ trivialize to the multiples of $1_A$, is given on fractals in [CGIS].
\end{rem}

\subsection{The Lipschitz algebra of a Dirichlet spaces}
Here we isolate a subalgebra of the Dirichlet algebra $\B$ which play the role of algebra of Lipschitz functions on a Riemannian manifold.
\begin{lem} 1. For $a\in \B$, the following conditions are equivalent:\\
i) the carr\'e du champ $\Gamma[a]$ is absolutely continuous with respect to the trace $\tau$ and its Radon-Nikodym derivative $\ds \frac{d\Gamma[a]}{d\tau}$ is bounded (which we write shortly $\Gamma[a]\in \M$)
\[
(\partial a|(\partial a)b)_\calH=\tau(b\Gamma[a])\qquad b\in A;
\]
ii) there exists a constant $C_a\ge 0$ such that
\[
|\langle\Gamma[a],b\rangle|\leq C_a \tau(|b|)\qquad b\in \B;
\]
iii) the vector $\partial a\in \calH$ is right-$\tau$-bounded, i.e. there exists a constant $C_a\ge 0$ such that
\[
\|\partial(a)b\|^2_\calH\leq C_a \tau(b^*b),\quad b\in \B.
\]
2. The set $\A_\E\subseteq\B$ whose elements and their adjoint satisfy the conditions above is a $*$-subalgebra of $\B$.
\end{lem}
\begin{proof} 1. is straightforward and 2. is a consequence of the symmetry and the Leibnitz rule of the derivation (\ref{Leibnitz}).
\end{proof}
\begin{defn}\label{AL} $\A_\E$ will be called the {\it Lipschitz algebra} of the Dirichlet space $(\E,\F)$. For $a\in \A_\E$, we shall denote $R(a):L^2(A,\tau)\to\calH$ the bounded operator characterized by
\[
R(a):L^2(A,\tau)\to\calH\qquad R(a)b:=(\partial a)b\qquad b\in \B\,.
\]
\end{defn}


\begin{rem}\label{partiala} Notice that due to the Leibnitz rule for $\partial$, one has, for $b\in \B$
\[
R(a)b=\partial(a)b=\partial(ab)-a\partial b=[\partial,a]b\]
\end{rem}
so that $a$ belongs to $\A_\E$ if and only if it has bounded commutator with the derivation $\partial$.

\subsection{The smooth subalgebra}
We show that the Lipschitz algebra $\A_\E$ contains a subalgebra of elements in the operator domain of the generator.
\begin{prop}\label{A2infty} Let $D_M(L)$ be the domain of the generator on the von Neumann algebra. Then the space
\[
\A_\E^{2,\infty}:=\A_\E\cap D_M(L)
\]
is a $*$-subalgebra of the Lipschitz algebra $\A_\E$.
\end{prop}
\begin{proof} Observe first that for $t>0$, since $(I+tL)^{-1}$ is a *-weakly continuous contraction of $M$ and $\E$ is symmetric with respect to $\tau$, $(I+tL)^{-1}$ will be also a contraction of the predual $L^1(A,\tau)=M_*$. Notice then that $\B^2\subset L^1(A,\tau)$ is dense in $L^1(A,\tau)$: in fact if $x\in M$ is orthogonal to $\B^2$, one has $0=\tau(bax)=(a^*|xb)_{L^2(A,\tau)}$ for all $a,b\in \B$ so that $x=0$. Hence $\B\cap L^1(A,\tau)$ is dense in $L^1(A,\tau)$. Consider now $a\in \A_\E^{2,\infty}$. According to [CS], for $b\in \B$, one has
\begin{equation*}
\begin{split}
2\langle\Gamma[a],b\rangle=(L^{1/2}(a)&|L^{1/2}(ab))_{L^2(A,\tau)}+(L^{1/2}(ab^*)|L^{1/2}a)_{L^2(A,\tau)}-(L^{1/2}(a^*a)|L^{1/2}b)_{L^2(A,\tau)}
\end{split}
\end{equation*}
so that, if $\Gamma[a]\in M$ and $L(a)\in M$, there exists a constant $C$ such that
\[
\big(L^{1/2}(a^*a)|L^{1/2}b)_{L^2(A,\tau)}\big|\leq C\, \tau(|b|)\,,\;b\in \B\cap L^1(A,\tau).
\]
In particular, for any $t>0$ and $b\in \B\cap L^1(A,\tau)$, $b\geq 0$\,:
\begin{equation*}
\begin{split}
\langle L(\frac{I}{I+tL}(a^*a)),b\rangle_{L^2(A,\tau)}&=\langle L^{1/2}(a^*a),L^{1/2}(\frac{I}{I+tL}(b))\rangle_{L^2(A,\tau)}\\
&\leq C\,\tau(\frac{I}{I+tL}(b))\leq C\tau(b)\,.
\end{split}
\end{equation*}
This estimate implies that $\ds \big\| L((I+tL)^{-1}(a^*a))\big\|_M\leq C$ is bounded uniformly on $t>0$. As $\ds\lim_{t\downarrow 0} (I+tL)^{-1}(a^*a)=a^*a$, $\sigma$-weakly in $M$, and $L$ is a $\sigma$-weakly closed operator on $M$, we have proved $a^*a\in D_M(L)$.
\end{proof}

\subsection{Examples by proper, conditionally negative type functions on discrete groups.}\label{group} Let $G$ be a discrete group with the Haagerup property and $\ell$ a proper, conditionally negative type function on $G$. The operator $L$ of multiplication by $\ell$ in $l^2(G)=L^2(C^*_{red}(G),\tau)$ is the generator of the Dirichlet form
\[
\E:l^2(G)\to [0,+\infty]\qquad \E[a]=\sum_{g\in G}|a(g)|^2.
\]
($\tau$ being the canonical trace on the reduced $C^*$-algebra of $G$).

Let $\lambda$ be the left regular representation of $A=C^*_{red}(G)$ and define $\A_\infty$ as the algebra of elements $a=\sum_{g\in G} a(g)\lambda(g)$ in $ C^*_{red}(G)$ whose sequence of Fourier coefficients has finite support ($\{g\in G\,,\;a(g)\not=0\}$ is finite). It is known that there exists a unitary representation $\pi$ of $G$ in on a Hilbert space $\mathfrak h$ and a $1$-cocycle
\[
c:G\to\mathfrak{h}.\qquad c(gg')=c(g)+\pi(g)c(g')
\]
such that the conditionally negative type definite function $\ell$ can be represented as
\[
\langle c(g'),c(g)\rangle_{\mathfrak h} = \frac{1}{2}\big(\ell(g)+\ell(g')-\ell(g'{^{-1}}g)\big)\,,\qquad \ell(g)=\|c(g)\|^2\,,\quad g,g'\in G.
\]
The tangent bimodule is then $\calH=\mathfrak h\otimes l^2(G)$ where $G$ acts on the left by the diagonal representation $\pi\otimes \lambda$ and acts on the right by $1_{\mathfrak h}\otimes \rho$ where $\rho$ is the right action of $G$ on $\ell^2(G)$. For $a\in \A_\infty$, the derivation representing $\E$ is given by
$$\partial a=\sum_G a(g)c(g)\otimes \delta_g\,,\;a\in \A_\infty$$
and the {\it carr\'e du champ} is indeed an element of $\A_\infty$
\[
\begin{split}
\Gamma[a]&=\sum_{g_1,g_2\in G} \overline{a(g_2)}\,a(g_1)\langle c(g_2),c(g_1)\rangle_{\mathfrak h} \lambda (g_2^{-1}g_1) \\
&=\frac{1}{2}\sum_{g_1,g_2\in G} \overline{a(g_2)}\,a(g_1)\big(\ell(g_2)+\ell(g_1)-\ell(g_2^{-1}g_1)\big)\lambda (g_2^{-1}g_1).
\end{split}
\]
So that $\A_\infty\subset \A_L^{2,\infty}\subset \A_L$, which implies that both $\A_L$ and $\A_L^{2,\infty}$ are dense subalgebras of $A$.

\begin{rem}\label{Gammalambda} Notice that $\Gamma[\lambda(g)]=\ell(g)\,I_{\ell^{\infty}(G)}$ is an invertible element of $\A_\infty$ whenever $\ell(g)\not=0$, i.e. for elements of $G$ outside a finite subset.

As a consequence, as $<b,\Gamma[a]b>=||R(a)b||^2$, one has
\begin{equation}\label{Rlambdag} R(\lambda(g))^*R(\lambda(g))=\ell(g)\,I_{\ell^\infty(G)}
\end{equation}
\end{rem}

\subsection{The spectral triple of a Dirichlet space.}

{\it From now on and till the end of this paper, we assume the generator $L$ of the Dirichlet form $\E$ to have discrete spectrum away from its kernel}. Let us consider the triple
$$(L^2(A,\tau)\oplus \calH\,,\, \A_\E\,,\,D)$$
where the Dirichlet algebra $\A_\E$, as a subalgebra of $A$, acts on $L^2(A,\tau)\oplus\calH$ by the diagonal action of $A$ on the left, both on $L^2(A,\tau)$ and $\calH$ and the {\it Dirac operator} is defined as
\[
D=\begin{pmatrix} 0 & \partial^* \\ \partial & 0 \end{pmatrix}\,.
\]

\begin{thm}\label{D}
The triple $(L^2(A,\tau)\oplus \calH\,,\, \A_\E\,,\,D)$ is an essentially discrete spectral triple. In particular
\vskip0.1truecm\noindent
i) the commutator of the derivation $\partial$ with the actions of $\A_\E$ on $L^2(A,\tau)$ and $\calH$ is given by
\[
[\partial ,a\,]=R(a)\quad\text{for all}\quad a\in \A_\E\quad \text{with}\quad R(a)b=(\partial a)b\quad\text{for all}\quad b\in\B;
\]
ii) for $a\in A_\E$ we have $\|[D,a]\|=\max(\|R(a)\|,\|R(a^*)\|)$ and
\[
\big[D,a\big]=\begin{pmatrix} 0 & [\partial^*,a] \\ [\partial,a] & 0 \end{pmatrix}=\begin{pmatrix} 0 & -R(a^*)^* \\ R(a) & 0 \end{pmatrix};
\]
iii) in the polar decomposition $\partial=u\,L^{1/2}$ of the derivation $\partial$, the partial isometry $u:L^2(A,\tau)\to\calH$ is such that $u^*u=I_{L^2(A\tau)}-p_0$ and $uu^* = I_\calH -q_0$, where $p_0$ and  $q_0$ are the orthogonal projections onto $ker(\partial)=ker(L)$ and $ker(\partial^*)=Im(\partial)^\perp$, respectively;
\vskip0.1truecm\noindent
iv) one has $\ds D^2=\begin{pmatrix} L & 0 \\ 0 & uLu^* \end{pmatrix}$ and $\ds |D|=\begin{pmatrix} L^{1/2} & 0 \\ 0 & uL^{1/2}u^* \end{pmatrix}
=\begin{pmatrix} u^*\partial & 0 \\ 0 & \partial u^* \end{pmatrix}$;
\vskip0.1truecm\noindent
v) if we enumerate $\lambda_1\leq \lambda_2 \leq \cdots \leq \lambda_n \leq \cdots$ the nonzero eigenvalues of $L$ , the corresponding enumeration for $|D|$ is
$\sqrt\lambda_1\leq \sqrt\lambda_1\leq \sqrt\lambda_2\leq \sqrt\lambda_2 \leq \cdots \leq \sqrt\lambda_n \leq \sqrt\lambda_n \leq \cdots$
i.e. $\lambda_n(|D|)=\lambda_{[(n+1)/2]}^{1/2}$ and $\mu_n(|D|^{-1})=\lambda_{[n/2]+1}^{-1/2}$ ($[r]$ being the integer part of a real $r$);
\vskip0.1truecm\noindent
vi) the projection onto the kernel of $D$ is $P_0=\begin{pmatrix} p_0 & 0 \\ 0 & q_0 \end{pmatrix}$\,.
\end{thm}
\begin{proof}
Straightforward by Lemma 3.3.
\end{proof}
On compact quantum groups, spectral triples of the above type has been constructed in \cite{CFK} Theorem 8.4, staring from the Dirichlet form of GNS-symmetric noncommutative Levy processes.

\begin{rem} In all interesting examples, $ker(\partial^*)$ is infinite dimensional and then, even if $L$ has a finite dimensional kernel (which often occurs),
$P_0$ may have, in general, infinite rank.
\end{rem}

\medskip \subsection{The Fredholm modules of a Dirichlet space}

According to subsection \ref{Fred} and with the notations of Theorem \ref{D}, the Fredholm operators associated to the Dirichlet spaec are
$$F_0=P_0+\frac{D}{|D|}=\begin{pmatrix} p_0 & u^* \\ u & q_0 \end{pmatrix}\;\text{ and }\,
F=P_0+\frac{D}{\sqrt{1+D^2}}=\begin{pmatrix} p_0 &\frac{I}{\sqrt{I+L}} \partial^* \\ \partial\frac{I}{\sqrt{I+L}} & q_0 \end{pmatrix}\,.
$$
They differ by an element in the ideal generated by $|D|^{-2}$ (cf. subsection \ref{Fred}).

\smallskip
Proposition \ref{Fa1}, Corollary \ref{muk} and point v) of Theorem \ref{D} lead to the following representation
\begin{prop}\label{FA3}
i) For $a=a^*\in \A_\E$ there exist bounded operators $\alpha$, $\beta$ and $\gamma$ such that
\[
i[F,a]=\alpha |D|^{-1}+|D|^{-1}\beta + |D|^{-1/2} \gamma\, |D|^{-1/2}
\]
with $||\alpha||\leq 2\,||\,[D,a]=2||R(a)||$, $||\beta||\leq 2\,||R(a)||$, $||\gamma||\leq ||\,R(a)\,||$;
\vskip0.1truecm\noindent
ii) $\ds \mu_{8k}([F,a]) \leq 5 \,\max(||R(a)||\,||R(a^*||)\,\lambda_{k+1}^{-1/2}$ for all $k\in\mathbb{N}$.
\end{prop}

\begin{rem} When the Lipschitz algebra is not dense in $A$, as in the case of harmonic forms on the p.c.f. fractals where $\A_\E$ reduces to constant functions, an alternative, natural choice for the Fredholm module is $(\calH, A, \widetilde F)$ where $\widetilde F=q_0^\perp -q_0$ is the orthogonal symmetry on $\calH$ with respect to the subspace of gradients $Im(\partial)=q_0(\calH)$. This is what has been done in [CS2] for post critically finite fractals].

This alternative choice leads to different estimates for the quantum derivative (cf. [CS2]). However, up to a sign, they lead to the same $K$-homology class, as shown in the following
\end{rem}
\begin{lem} If $L$ has discrete spectrum and $\A_\E$ is dense in $A$, the Fredholm modules $(L^2(A,\tau)\oplus \calH,A,F)$ and $\big(L^2(A,\tau)\oplus\calH,A,I\oplus(-\widetilde F)\big)$ are homotopic.
\end{lem}
\begin{proof} We forget about $p_0$ which is finite dimensional. $(L^2(A,\tau)\oplus \calH,A,F)$ and $(L^2(A,\tau)\oplus \calH,A,F_0)$ are obviously homotopic (through the family  $(L^2(A,\tau)\oplus \calH,A,(1-t)F+t\,F_0)$, $t\in [0,1]$, while  $(L^2(A,\tau)\oplus \calH,A,F_0)$ and $\big(L^2(A,\tau)\oplus\calH,A,I\oplus(-\widetilde F)\big)$ are homotopic through the family of Fredholm operators
\[
\begin{pmatrix} \sin\,\theta\,I & \cos\,\theta\,u^* \\ \cos\,\theta u & (1+\sin\,\theta)q_0-\sin\,\,\theta\,I\,\end{pmatrix}\qquad \theta\in [0,\pi/2].
\]
\end{proof}

\medskip \subsection{Commutators with elements of the smooth subalgebra.}

In this subsection, we start with some selfadjoint element $a\in \A_L^{2,\infty}$. The goal is to prove that the commutator $[\,|D|,a\,]$ is bounded, so that Proposition \ref{commcomm} applies. We start with some intermediary results.
\begin{lem}\label{La} $\ds [L,a]\frac{1}{\sqrt{1+L}}$ is a bounded operator.
\end{lem}
\begin{proof} Let us compute, for $c\in A$ and $b\in Dom_{L^2}(L)\cap \B$ and making use of the rules established in [CS, square roots] and making use of the symmetry of the $A$-$A$-bimodule $\calH$\,:
\begin{equation*}\begin{split}
(c|[L,a]b)&=\tau(c^*(L(ab)-aL(b)))=\tau(c^*(-2\Gamma(a^*,b)+L(a)b)\\&=-2(\partial(a^*)c,\partial b)+(c,L(a)b)\\&=-2(c,R(a^*)^*\partial b)+(c,L(a)b)
\end{split}\end{equation*}
from which we deduce
\begin{equation}
[L,a]=-2R(a^*)^*\partial +L(a)\,.
\end{equation}
\end{proof}

\medskip As a consequence, we establish the following:
\begin{lem}\label{Lgamma} For $\gamma\in (0,1/2)$, $(I+L)^{1/2-\gamma}\,[(I+L)^\gamma,a\,]$ is a bounded operator.

As a consequence, $[(I+L)^\gamma,a\,]$ is a compact operator.
\end{lem}

\noindent {\it Preuve\,:}
We start with
\begin{equation*}\begin{split}
\big[(1+L)^{\gamma}\,,\,a\,\big]&=C_\gamma\,\int_0^{+\infty} t^{\gamma-1} \left[ \frac{1+L}{t+1+L}\,,\,a\right]dt\\
&=-C_\gamma\,\int_0^{+\infty} t^{\gamma-1} \left[ \frac{t}{t+1+L}\,,\,a\right]dt\\
&=C_\gamma\,\int_0^{+\infty} t^{\gamma}\frac{1}{t+1+L}\,[L,a]\,\frac{1}{t+1+L}\,dt\\
&=C_\gamma\,\int_0^{+\infty} t^{\gamma}\frac{1}{t+1+L}\,[L,a]\frac{1}{\sqrt{1+L}}\;\frac{\sqrt{1+L}}{t+1+L}\,dt\,.\\
\end{split}\end{equation*}
Coupling with  $x,y\in L^2(A,\tau)$, Lemma \ref{La} provides some constant $C'$ such that\,:
\begin{equation*}\begin{split}
\Big|\big<\,y&,\,\big[(1+L)^{\gamma}\,,\,a\,\big]\,x\,\big>\Big| \leq \, C'\int_0^{+\infty} t^{\gamma} \;\Big\|\frac{1}{t+1+L}\,y\Big\|\;\Big\|\frac{\sqrt{1+L}}{t+1+L}\,x\,\Big\|\,dt\\
&\leq C' \left(\int_0^{+\infty} t^{2\gamma} \big<\,y\,,\, \frac{1}{(t+1+L)^2}\,y\big>\,dt\right)^{1/2}\;\left(\int_0^{+\infty}  \big<\,x\,,\, \frac{1+L}{(t+1+L)^2}\,x\big>\,dt\right)^{1/2}\,.
\end{split}\end{equation*}
One checks easily that $\ds \int_0^{+\infty} t^{2\gamma}\frac{1}{(t+1+L)^2}dt$ is proportional to $(1+L)^{2\gamma-1}$, and that \\ $\ds \int_0^{+\infty} \frac{1+L}{(t+1+L)^2}dt$ is the identity operator.

From which $\ds \Big|\big<\,y,\,\big[(1+L)^{\gamma}\,,\,a\,\big]\,x\,\big>\Big| \leq C'' ||(1+L)^{\gamma-1/2}y\,||\;||x||
$ and the result \hfill $\square$

\medskip As. a corollary, we get
\begin{lem} $[\sqrt{1+L}\,,a\,]$ is a bounded operator.
\end{lem}
\begin{proof} Apply Lemma \ref{Lgamma} with $\gamma=1/4$\,: $(I+L)^ {1/4}[(I+L)^{1/4},a]$ and $[(I+L)^{1/4},a](I+L)^{1/4}$ are bounded operators (for the latter, consider the adjoints). The Leibnitz rule provides
$$[(1+L)^{1/2},a]=(1+L)^{1/4}\,[\,(1+L)^{1/4},a\,]\,+\,[(1+L)^{1/4},\,a\,]\,(1+L)^{1/4}$$
which is a bounded operator. \end{proof}

We can now prove the main result of this section\,:

\begin{prop}
For $a\in A_L^{2,\infty}$  the operator $[\,|D|,a]$ is a bounded. Hence, the conclusions of Theorem \ref{commcomm} hold true for $a$.
\end{prop}
\begin{proof} Let $\partial = u L^{1/2}$ be the polar decomposition of $\partial$.

As $[\partial ,a]=[u,a|\,L^{1/2}+u\,[a,L^{1/2}]$ and $u\,[a,L^{1/2}]$ are bounded,  $[u,a|\,L^{1/2}$ is a bounded operator. We have now
\begin{equation*}\begin{split} \big[\,(\partial\partial^*)^{1/2},a\big] &= \big[ uL^{1/2}u^*,a\big]\\
&=[\partial,a]u^*+uL^{1/2}\,[u^*,a] \\
&=[\partial,a]u^*+u\big([a^*,u]L^{1/2}\big)^*
\end{split}\end{equation*}
which is a bounded operator. This ends the proof.
\end{proof}

\section{Lower bounds on singular values of quantum differentials}
In this section we consider conditions providing lower bounds for the singular values $\mu_k(i[F,a])$. We also propose general and particular examples where these conditions are satisfied.

\subsection{Lower bounds on quantum differentials of Dirichlet spaces}

\begin{prop}\label{muk2}
Suppose that $a\in A_L^{2,\infty}$ satisfies the three conditions\,:
\begin{itemize}
\item[i)] The commutator $[\sqrt{I+L},a]$ is a compact operator.
\item[ii)] $R(a)^*R(a)$ has a finite dimensional kernel
\item[iii)] $R(a)^*R(a)$ is invertible on ${\rm ker}(R(a))^\perp$.
\end{itemize}
Then there exists an integer $k_0$ and a constant $C(a)$ such that
$$\mu_k(i[F,a])\geq C(a)\, (1+\e(k))\,\lambda_{k+k_0}(L)^{-1/2}\,$$
where $C(a)$ is the infimum of the spectrum of $|R(a)|$ on $ker(R(a))^\perp$, $\e(k)$ is a sequence tending to $0$ as $k\to \infty$ and $k_0=dim(ker(L))+1$.
\end{prop}
\begin{proof} Let us start with an observation\,: as
$$\begin{pmatrix} 0 & 0 \\ [\partial\frac{I}{\sqrt {I+L}},a] & 0\end{pmatrix} = \begin{pmatrix} 0 & 0 \\ I & 0 \end{pmatrix}
 \begin{pmatrix} p_0 & [\frac{I}{\sqrt {I+L}}\partial^*,a] \\ [\partial\frac{I}{\sqrt {I+L}},a] & q_0\end{pmatrix}
 \begin{pmatrix} 0 & 0 \\ 0 & I \end{pmatrix}
$$
which provides
$$\mu_k([F,a])\geq \mu_k\big([\partial\frac{I}{\sqrt {I+L^2}},a]\big)\,.$$
We have then
$$[\partial\frac{I}{\sqrt {I+L}},a]=[\partial,a]\frac{I}{\sqrt {I+L}}-\partial\frac{I}{\sqrt {I+L}}\big[ \sqrt{I+L}\,,a\big]\frac{I}{\sqrt {I+L}}
$$
i.e., since $[\partial,a]$ is equal to $R(a)$, $\ds \partial\frac{I}{\sqrt {I+L}}$ is bounded and $\big[ \sqrt{I+L}\,,a\big]$ is compact, we get
$$[\partial\frac{I}{\sqrt {I+L}},a]=\big(R(a)+\kappa\big)\frac{I}{\sqrt {I+L}}
$$
with $\kappa$ compact. Applying Lemma \ref{Tsigma}, the thesis follows.
\end{proof}

In the sequel, we show that the conditions of the above result are realistic. In particular, according to equation \ref{Rlambdag} in Remark \ref{Gammalambda}, condition ii) is satisfied whenever $\E$ is the Dirichlet form associated with a proper negative type function $\ell$ on a discrete group $G$ and $a=\lambda(g)$, $g\in G$, $\ell(g)\not=0$.

\subsection{Roots of generators} (\cite{CS1}).

Suppose that $\widetilde \E$ is a (symmetric, regular, completely) Dirichlet form on $L^2(A,\tau)$ with generator $\widetilde L$
such that $\A_{\widetilde L}^{2,\infty}$ is dense in $A$, and consequently $\A_{\widetilde L}$ is dense in $A$. Fix any $\beta\in (0,1)$, let $\E$ be the Dirichlet form with generator $L=\widetilde L^\beta$ (cf. \cite{CS1} Section 10.4). The semigroup $e^{-tL}$ is called {\it subordinated} to the semigroup $e^{-t{\widetilde L}}$ (see \cite{C1}).

\begin{lem}
We have $A_{\widetilde L}^{2,\infty}\subset A_L^{2,\infty}$\,.
\end{lem}
\begin{proof} We start with the standard formula for roots of positive operators\,:
\begin{equation}\label{root} \begin{split}
\widetilde L^\beta=C_\beta \int_0^{+\infty} t^{\beta-1} \frac{\widetilde L}{t+\widetilde L}dt= C_\beta\int_0^{+\infty} s^{-\beta} \frac{\widetilde L}{I+s\widetilde L}ds
\end{split}\end{equation}
with $C_\beta=\sin(\beta \pi)/\beta \pi$\,.
As $\ds \frac{I}{I+s\widetilde L}$ acts as a completely positive contraction of $M$, for $a\in \A_{\widetilde L}^{2,\infty}$ we have on one hand
$$\big\| \frac{\widetilde L}{I+s\widetilde L}(a)\big\|_M|\leq \| \widetilde L(a)\|_M$$
so that the integral converges in $M$ for $s\to 0$, and on the other hand
$$\big\| \frac{\widetilde L}{I+s\widetilde L}(a)\big\|_M=\frac{1}{s}\big\| a-\frac{I}{I+s\widetilde L}(a)\big\|_M\leq \frac{2}{s}||a||_M$$
so that the integral converges in $M$ as $s\to +\infty$. We have proved $\A_{\widetilde L}^{2,\infty}\subset D_M(L)$. As $\A_{\widetilde L}^{2,\infty}$ is an involutive algebra, for $a\in \A_{\widetilde L}^{2,\infty}$ we have also $a^*\in D_M(L)$ and $a^*a\in D_M(L)$, so that $\Gamma[a]=\frac{1}{2}\big(a^*L(a)+L(a)^*a-L(a^*a)\big)\in M$. Which proves $\A_{\widetilde L}^{2,\infty}\subset \A_L$.
\end{proof}

\begin{cor}\label{Lbeta}  With $L=\widetilde L^\beta$ as in formula (\ref{root}), there exists an involutive subalgebra $\A_\infty$ of $\A_L^{2,\infty}$ dense in $A$ such that, for any $a\in \A_\infty$, the commutator $[\sqrt{I+L},\,a]$ is a compact operator.
\end{cor}
\begin{proof} Take $\A_\infty=\A_{\widetilde L}^{2,\infty}$ and apply Lemma \ref{Lgamma} with $\gamma=\beta/2$.
\end{proof}

\section{Slow conditionally negative type function on discrete groups}\-
In this section we come back to the framework of subsection \ref{group}. Thus, $G$ is a discrete group with the Haagerup property, $\tau$ is the canonical trace on $C^*_{red}(G)$, $\ell$ is a proper negative type function on $G$, $L$ is the operator of multiplication by $\ell$ in $\ell^2(G)=L^2(C^*_{red}(G),\tau)$  and $\E_\ell$ is the Dirichlet form with generator $L$. We recall that $\A_\infty$ is the dense $*$-subalgebra of $A=C^*_{red}(G)$ of the $a=\sum_{g\in G}a(g)\lambda(g)$ with finite support.

\subsection{Asymptotic orthogonality for 1-cycles} Let us start with a simple observation: for fixed $g\in G$ we have
\[
\ell(g^{-1}g')=\ell(g)+\ell(g')-2(c(g)|c(g')) = \ell(g')+O(\ell(g'))^{1/2}\qquad \text{as}\,\, g'\to\infty\,.
\]
We are going to show that Proposition \ref{muk2} applies to the Dirichlet forms $\E_\ell$ provided the negative type function $\ell$ satisfies a strengthened form of the above asymptotic  estimate.

\begin{defn} (Slow conditionally negative type functions)
A conditionally negative type function $\ell:G\to[0,+\infty)$ is said to be {\it slow} if it is proper and if, for any fixed $g\in G$,
\begin{equation}\label{o}
\ell(g^{-1}g')=\ell(g')+o(\ell(g'))^{1/2}\qquad \text{as}\,\, g'\to\infty\,.
\end{equation}
\end{defn}

\begin{prop}\label{condneg} Let $\ell:G\to0,+\infty)$ be slow conditionally negative type function. Then, Proposition \ref{muk2} applies to the Dirichlet form $\E_\ell$ for any $a=\lambda(g)$ with $\ell(g)\not=0$ (i.e. all $g\in G$ but a finite number),
so that there exists a sequence $\e(k)\to 0$ such that
$$\sqrt{\ell(g)}\,(1+\e(k))\,\lambda_{k+1}(L)^{-1/2} \leq \mu_k\big(i\big[F,\lambda(g)\,\big]\big) \leq 5\,\sqrt{\ell(g)}\, \lambda_{[k/8]+1}(L)^{-1/2}\,.$$
\end{prop}
\begin{proof} One checks easily that for such $a=\lambda(g)$, $[\sqrt{I+L},a]=k_g\lambda(g)$ where $k_g$ is the multiplication operator by the function $g'\to \sqrt{1+l(g^{-1}g')}-\sqrt{1+\ell(g')}$. Check that
\begin{equation*}\begin{split}  \sqrt{1+l(g^{-1}g')}-\sqrt{1+\ell(g')}&=
\frac{\ell(g^{-1}g')-\ell(g')}{\sqrt{1+l(g^{-1}g')}+\sqrt{1+\ell(g')}}\\
&=\frac{o(\ell(g')^{1/2})}{\sqrt{1+l(g^{-1}g')}+\sqrt{1+\ell(g')}}
\end{split}\end{equation*}
tends to $0$ as $g'\to \infty$, so that $k_g$ is a compact operator. Condition (i) of Proposition \ref{muk2} is satisfied. Condition (ii) of Proposition \ref{muk2} is provided by identity (\ref{Rlambdag}) of Remark \ref{Gammalambda}. The estimates come from Proposition \ref{muk2}, Corollary \ref{muk}, Conclusion 5 of Theorem \ref{D} and identity (\ref{Rlambdag}) of Remark \ref{Gammalambda}.
\end{proof}

\begin{cor} Let $\widetilde \ell$ be a proper conditionally negative type function on $G$. Then $\ell:=\widetilde \ell^\beta$ is a slow conditionally negative type function, for an arbitrary $\beta\in (0,1)$ and Proposition \ref{muk2} applies to the Dirichlet form $\E_\ell$.
\end{cor}
\begin{proof} Fix $g$ and make $g'$ tend to $\infty$. As noticed above, one has $\widetilde\ell(g^{-1}g')=\widetilde\ell(g')+O(\widetilde\ell(g')^{1/2})=\widetilde\ell(g')\big(1+O(\widetilde\ell(g')^{-1/2}\big)$ so that
\begin{equation*}\begin{split}
\ell(g^{-1}g'))=\widetilde\ell(g^{-1}g')^\beta&=\widetilde\ell(g')^\beta(1+O(\widetilde\ell(g')^{-1/2})=\ell(g')+o(\ell(g'))\,.
\end{split}\end{equation*}
\end{proof}

\begin{cor}\label{freegroup} Suppose that $G=\mathbb{F}_p$ is the free group with $p$ generators and $\ell$ is the length function, which is conditionally negative (cf. \cite{Haa}). Then  Proposition \ref{muk2} applies to the Dirichlet form $\E_\ell$ for $a=\lambda(g)$ with $g\neq e$ so that there exist constants $C_1, C_2>0$ and an integer $k_0$ such that
\[
C_1 ({\rm Log (k)})^{-1/2} \leq \mu_k(i[F,a]) \leq C_2 ({\rm Log (k)})^{-1/2}\,,\;k\geq k_0\,
\]
with $C_1$ (resp. $C_2$) arbitrarily close to $\sqrt{\ell(g)}$ (resp. $5\sqrt{\ell(g)}$) if $K_0$ is chosen large enough.
\end{cor}
\begin{proof} The first assumption is straightforward: since $\ell$ is a length function, we have $|\ell(s^{-1}t)-\ell(t)|\leq \ell(s)$ so that $\ell$ is a proper, slow, conditionally negative type function. For the second assumption, check that the ball $B_m$ of radius $m$ in $G$, $B_m=\{g\in G\,|\, \ell(g)\leq m\}$, has a cardinality $\ds |B_m|=1+\frac{p}{p-1}(2p-1)^m$. As $\lambda_k=m$ for $|B_k|\leq m<\|B_{k+1}|$, we get $\ds \lambda_k\sim \frac{k}{{\rm Log(2p-1)}}$, $k\to \infty$. Corollary \ref{muk} and Proposition \ref{condneg} provide the result.
\end{proof}

\subsection{Weight functions as slow conditionally negative type functions}
Here we prove that all {\it weight functions} on discrete groups are slow negative type functions.\\
Let $\ell$ be a conditionally negative type function on the group $G$ (symmetric and strictly positive away from the unit) and let $(\mathfrak{h}_\ell,\pi_\ell,c)$ be the associated 1-cocycle $c:G\to \mathfrak{h}_\ell$ with
\begin{equation}
c(st)=c(s)+\pi_\ell(s)c(t)\,,\quad ||c(t)||^2 = \ell(t)\,,\quad (c(s)|c(t))_H =\undemi \big(\ell(s)+\ell(t)-\ell(s^{-1}t)\big)\,.
\end{equation}
One checks easily that the vectors $c(s), c(t)$ in $\mathfrak{h}_\ell$ are orthogonal if and only if {\it $e$ lies between $s$ and $t$}, i.e. $\ell(s^{-1}t) = \ell (s) + \ell (t)$. The function $\sqrt{\ell}$ is conditionally negative too and provides a left-invariant metric on $G$ by
\[
d_{\sqrt\ell} (s,t):=\sqrt{\ell({s^{-1}t})}\qquad s,t\in G\, .
\]
The cocycle is an isometric embedding of the metric space $(G,d_{\sqrt\ell})$ into the Hilbert space $\mathfrak{h}_\ell$
\[
\begin{split}
\|c(t)-c(s)\|_{\mathfrak{h}_\ell} &= \|c(ss^{-1}t)-c(s)\|_{\mathfrak{h}_\ell} = \|c(s)+\pi_\ell (s)(c(s^{-1}t))-c(s)\|_{\mathfrak{h}_\ell} \\
&= \|c(s^{-1}t))\|_{\mathfrak{h}_\ell} = \sqrt{\ell({s^{-1}t})} = d_{\sqrt\ell} (s,t)\qquad s,t\in G\, .
\end{split}
\]
The characteristic property of a slow negative type function, expressed as
\[
\ell(s)+\ell(t)-\ell(s^{-1}t) = o({\sqrt{\ell (t)} })\qquad t\to\infty,
\]
for each fixed $s\in G$, can be restated as
\[
0=\lim_{t\to \infty} \frac{\ell(s)+\ell(t)-\ell(s^{-1}t)}{2\sqrt{\ell (s)}\sqrt{\ell (t)}} = \lim_{t\to \infty}\frac{(c(s)|c(t))_{\mathfrak{h}_\ell}}{\|c(s)\|_{H_\ell}\cdot\|c(t)\|_{\mathfrak{h}_\ell}},\qquad s\in G.
\]
The property thus refers to the {\it asymptotic orthogonality} of $t\in G$ with respect to any fixed $s\in G$, when these are embedded in the Hilbert space $\mathfrak{h}_\ell$.
\vskip0.2truecm\noindent
We recall that a {\it weight} on a discrete group $G$ (\cite{deH}) is a negative type function such that
\[
\ell :G\to [0,+\infty)\qquad \ell (st)\le \ell (s) + \ell (t).
\]
\begin{lem}
Any {\it proper weight function} on a discrete group $G$ is a slow, conditionally negative negative type function.
\end{lem}
\begin{proof}
Since a weight satisfies $|\ell (s)-\ell (t)|\le \ell(s^{-1}t)$,
we have $0\le\ell (s)+\ell (t) - \ell(s^{-1}t)\le 2(\ell (s)\wedge\ell (t))$ and
\[
\frac{\ell(s)+\ell(t)-\ell(s^{-1}t)}{2\sqrt{\ell (s)}\sqrt{\ell (t)}} \le\frac{\ell (s)\wedge\ell (t)}{\sqrt{\ell (s)}\sqrt{\ell (t)}}.
\]
Since $\ell$ is proper we have
\[
\lim_{t\to \infty} \frac{\ell(s)+\ell(t)-\ell(s^{-1}t)}{2\sqrt{\ell (s)}\sqrt{\ell (t)}} \le \lim_{t\to \infty}  \frac{\ell (s)\wedge\ell (t)}{\sqrt{\ell (s)}\sqrt{\ell (t)}} = \lim_{t\to \infty} \sqrt{\frac{\ell (s)}{\ell (t)}} =0\, .
\]
\end{proof}
\begin{rem}
A weight function gives rise to a left-invariant metric on $G$: $d_\ell(s,t):=\ell (s^{-1}t)$.
If $k:=\inf\{\ell (t):t\in G,\,\, t\neq e\}$, the cocycle is a Lipschitz embedding of the metric space $(G,d_{\ell})$ into the real Hilbert space ${\mathfrak{h}_\ell}$
\[
\|c(t)-c(s)\|_H = \sqrt{\ell({s^{-1}t})} \le \sqrt{k^{-1}\ell^2({s^{-1}t})} = \frac{1}{\sqrt k} \ell({s^{-1}t}) = \frac{1}{\sqrt k} d_\ell (s,t)\qquad s,t\in G\, .
\]
\end{rem}
Examples include length functions of free groups $\mathbb{F}_n$ and the Heisenberg group.

\normalsize
\begin{center} \bf REFERENCES\end{center}

\normalsize
\begin{enumerate}

\bibitem[AHK]{AHK} S. Albeverio, R. Hoegh-Krohn, \newblock{Dirichlet forms and Markovian semigroups on C$^*$--algebras}, \newblock{\it Comm. Math. Phys.} {\bf 56} {\rm (1977)}, 173-187.


\bibitem[C1]{C1} F. Cipriani, \newblock{Dirichlet forms and Markovian semigroups on standard forms of von Neumann algebras},
\newblock{\it J. Funct. Anal.} {\bf 147} {\rm (1997)}, 259-300.

\bibitem[C2]{C2} F. Cipriani, \newblock{``Dirichlet forms on Noncommutative spaces''}, \newblock{Springer ed. L.N.M. 1954, 2007}.



\bibitem[CFK]{CFK} F. Cipriani, U. Franz, A. Kula, \newblock{Symmetries of L\'evy processes on compact quantum groups, their Markov semigroups and potential theory},
\newblock{\it J. Funct. Anal.} {\bf 266} {\rm (2014)}, no. 5, 2789--2844.


\bibitem[CS1]{CS1} F. Cipriani, J.-L. Sauvageot, \newblock{Derivations as square roots of Dirichlet forms}, \newblock{\it J. Funct. Anal.} {\bf 201} {\rm (2003)}, no. 1, 78--120.
%
\bibitem[CS2]{CS2} F. Cipriani, J.-L. Sauvageot, \newblock{Fredholm modules on P.C.F. self-similar fractals and their conformal geometry}, \newblock{\it Comm. Math. Phys.} {\bf 286} {\rm (2009)}, no. 2, 541--558.
%
%
%
\bibitem[CS3]{CS3} F. Cipriani, J.-L. Sauvageot, \newblock{Measurability, spectral densities and hypertraces in Noncommutative Geometry}, to appear in 
\newblock{\it Journal of Noncommutative Geometry}.
%
%



%
%
%
%
%
%
\bibitem[Co]{Co} A. Connes, \newblock{``Noncommutative Geometry''}, \newblock{Academic Press, New York, 1994}.

\bibitem[DL]{DL} E.B. Davies, J.M. Lindsay, \newblock{Non--commutative symmetric Markov semigroups}, \newblock{\it Math. Z.} {\bf 210} {\rm (1992)}, 379-411.
%
\bibitem[deH]{deH} P. de la Harpe, \newblock{``Topics in Geometric Group Theory''},\\ \newblock{Chicago Lectures in Mathematics, The University of Chicago Press, 2000}.
%
%
%

\bibitem[GK]{GK} I.C. Gohberg, M.G. Krein, \newblock{``Introduction to the theory of linear nonselfadjoint operators''}, \newblock{Transl. Math. Monogr., 18, Amer. Math. Soc., Providence, R.I., 1969;}.
%
\bibitem[Haa]{Haa} U. Haagerup, \newblock{An example of a nonnuclear C$^*$-algebra, which has the metric approximation property}, \newblock{\it Invent. Math.} {\bf 50} {\rm (1978)}, no. 3, 279-293.

\bibitem[Hil]{Hil} M. Hilsum, \newblock{Signature operator on Lipschitz manifolds and unbounded Kasparov bimodules}, \newblock{\it Lecture Notes in Math., 1132 Operator algebras and their connections with topology and ergodic theory (Busteni, 1983)}{\rm (1985)}, 254-288.
%
%
%
%
%
%
%
%

\bibitem[S]{S} J.-L. Sauvageot, \newblock{Quantum Dirichlet forms, differential calculus and semigroups,  Quantum Probability and Applications V},
\newblock{\it Lecture Notes in Math.} {\bf 1442} {\rm (1990)}, 334-346

\bibitem[SWW]{SWW} E. Schrhoe, M. Walze, J.-M. Warzecha, \newblock{Construction de triplets spectraux a partir de modules de Fredholm},
\newblock{\it C. R. Acad. Sci. Paris Ser. I Math.} {\bf 326} {\rm (1998)}, 1195--1199.

%
%
\bibitem[V]{V} D. Voiculescu, \newblock{On the existence of quasicentral approximate units relative to normed ideals. I.}, \newblock{\it  J. Funct. Anal.} {\bf 91} {\rm (1990)}, no. 1, 1-36.


\end{enumerate}
\end{document}